\documentclass[12pt]{article}
\usepackage[a4paper, margin=2.4cm]{geometry}
\usepackage{amssymb,amsthm, amsmath, mathtools}
\usepackage{hyperref}
\usepackage{graphicx}

\newtheorem{theorem}{Theorem}[section]
\newtheorem{corollary}[theorem]{Corollary}
\newtheorem{lemma}[theorem]{Lemma}
\theoremstyle{definition}
\newtheorem{definition}[theorem]{Definition}

\newtheorem{construction}[theorem]{Construction}

\newcommand{\R}{\mathbb R}
\newcommand{\proj}{\mathbb{PR}^2}
\newcommand{\RNum}[1]{\uppercase\expandafter{\romannumeral #1\relax}}

\renewcommand\footnotemark{}

\begin{document}
\title{On the number of ordinary circles}
\author{Hossein Nassajian Mojarrad \and Frank de Zeeuw\footnote{\hspace{-20pt} Both authors supported by Swiss National Science Foundation Grants 200020-144531 and 200021-137574.}
}
\date{}
\maketitle

\begin{abstract}
We prove that any $n$ points in $\R^2$, not all on a line or circle, determine at least $\frac{1}{4}n^2-O(n)$ ordinary circles (circles containing exactly three of the $n$ points).
The main term of this bound is best possible for even $n$.
Our proof relies on a recent result of Green and Tao on ordinary lines.
\end{abstract}


\section{Introduction} 
The classical Sylvester-Gallai theorem states that any finite non-collinear point set in $\R^2$ spans at least one \emph{ordinary line} (a line containing exactly two of the points).
A more sophisticated statement is the Dirac-Motzkin conjecture, according to which every non-collinear set of $n>13$ points in $\R^2$ determines at least $n/2$ ordinary lines. 
This conjecture was proved in 2013 by Green and Tao \cite{GT} for all $n$ larger than some fixed threshold $N_{GT}$ (see Section \ref{sec:greentao}).

It is natural to ask the corresponding question for \emph{ordinary circles} (circles that contain exactly three of the given points);
see for instance Section 7.2 in \cite{BMP} or Chapter 6 of \cite{KW}.
Elliott \cite{E} introduced this question in 1967, and proved that any $n$ points, not all on a line or circle, determine at least\footnote{When we say that $f(n)$ is at least $g(n)-O(h(n))$, we mean that there is a constant $C$ such that for all $n$ we have $f(n)\geq g(n) - C\cdot h(n)$.} $\frac{2}{63}n^2-O(n)$ ordinary circles. He suggested, cautiously, that the optimal bound is $\frac{1}{6}n^2-O(n)$.
Elliott's result was improved by B\'alintov\'a and B\'alint \cite{BB} to\footnote{This bound can be found in a remark at the very end of \cite{BB}.} $\frac{11}{247}n^2-O(n)$ in 1994, and Zhang \cite{Z} obtained $\frac{1}{18}n^2-O(n)$ in 2011.

We use the result of Green and Tao \cite{GT} to prove the following theorem, providing an essentially tight lower bound for the number of ordinary circles determined by a large point set. It solves Problem 7.2.6 in \cite{BMP} (and disproves Elliott's suggested bound).

\begin{theorem}\label{thm:main}
If $n$ points in $\R^2$ are not all on a line or circle,
then they determine at least $\frac{1}{4}n^2-O(n)$ ordinary circles.
\end{theorem}

If $n$ is even, the main term $\frac{1}{4}n^2$ is best possible, as shown by Construction \ref{constr:even} below (which is based on a construction in \cite{Z}).
But for odd $n$, the best upper bound that we know for the main term is $\frac {3}{8}n^2$, provided by Construction \ref{constr:odd}.
It remains an open problem to determine the right coefficient for odd $n$.

We have not attempted to determine the linear term in the bound more precisely, as it depends on the threshold $N_{GT}$ in the Green-Tao theorem, which is currently quite large.
The steps in our proof would be easy to quantify, given a better threshold.

\begin{construction}\label{constr:even}
Let $n$ be even. 
Take any two concentric circles $C_1$ and $C_2$. 
Let $S_1$ and $S_2$ be the vertex sets of regular polygons of size $n/2$ on $C_1$ and $C_2$ that are ``aligned'' in the sense that their points lie at the same set of angles from the common center (see Figure \ref{fig:concentric}).
  
 Any ordinary circle determined by $S_1\cup S_2$ contains two points from one circle and one from the other.
Because of the symmetry of the configuration, 
given $p,q\in S_1$ and $r\in S_2$,
the circle spanned by $p,q,r$ passes through another point of $S_2$,
unless $r$ is equidistant from $p$ and $q$.
On the other hand, if $r\in S_2$ is equidistant from $p,q\in S_1$, then  $p,q,r$ determine an ordinary circle.

First suppose $n/2$ is odd.
Then all $\binom{n/2}{2}$ pairs of points from $S_1$ have exactly one equidistant point in $S_2$, which gives $\binom{n/2}{2}$ ordinary circles with two points from $S_1$.
Similarly, there are $\binom{n/2}{2}$ ordinary circles with two points from $S_2$.  
Thus $S_1\cup S_2$ determines $2\binom{n/2}{2} = \frac{1}{4}n^2-O(n)$ ordinary circles.

Now suppose $n/2$ is even.
Then there are $\frac{1}{2}\binom{n/2}{2}$ pairs of points from $S_1$ that have no equidistant point in $S_2$, so these give no ordinary circles. There are $\frac{1}{2}\binom{n/2}{2}$ pairs of points from $S_1$ that have exactly two equidistant points in $S_2$, so these give $\binom{n/2}{2}$ ordinary circles.
Counting these circles for $S_1$ and $S_2$ gives $2\binom{n/2}{2} = \frac{1}{4}n^2-O(n)$ ordinary circles.
\end{construction}

\begin{figure}[ht]\label{fig:concentric}
\centerline{
\includegraphics[width=0.4\textwidth]{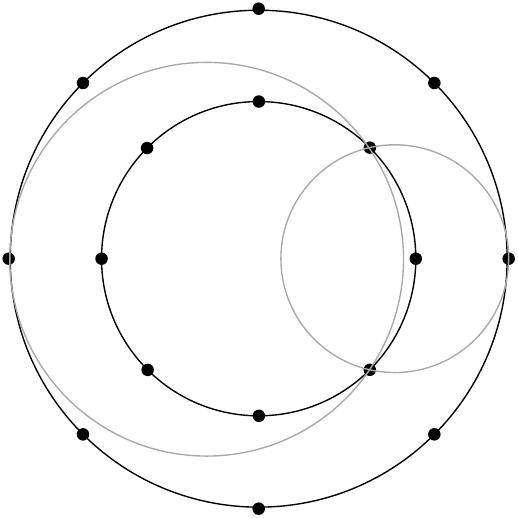}
\hspace{50pt}
\includegraphics[width=0.436\textwidth]{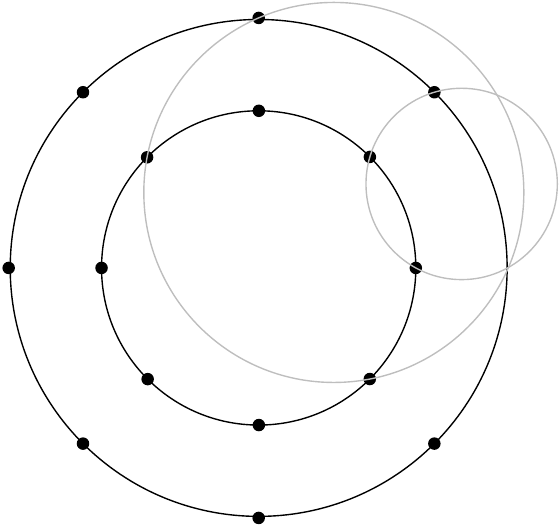}}
\caption{\small \sf Constructions \ref{constr:even} and \ref{constr:odd}, each with two ordinary circles.}
\end{figure}

\begin{construction}\label{constr:odd}
Let $n$ be odd. Take Construction \ref{constr:even} with $n+1$ points and remove an arbitrary point. 
The removal breaks the $O(n)$ ordinary circles that passed through the removed point, and it creates $\frac{1}{8}n^2-O(n)$ new ordinary circles.
Indeed, 
any four-point circle through the removed point becomes an ordinary circle.
The number of such circles is $\binom{n/2}{2}-O(n)$, since any pair of points on the circle not containing the removed point determines a circle through the removed point and another point,
except for the $O(n)$ pairs that are equidistant from the removed point, and possibly a pair that is collinear with the removed point.
Thus $S_1\cup S_2$ determines $\frac{3}{8}n^2-O(n)$ ordinary circles.
\end{construction}

\begin{construction}\label{constr:linecircle}
There are related constructions that lie on a line and a circle.
For each of the constructions above, 
pick a point on one of the circles that is not in the point set.
Apply circle inversion with this point as the center (see Section \ref{sec:inversion}).
The resulting point set lies on a line and a circle.
Every ordinary circle after the inversion corresponds to an ordinary circle or a three-point line before the inversion.
It is not hard to check that Construction \ref{constr:even} has no three-point lines, 
and Construction \ref{constr:odd} has one (the line through the removed point and its antepode).
Thus these constructions also have $\frac{1}{4}n^2-O(n)$ ordinary circles when $n$ is even and $\frac{3}{8}n^2-O(n)$ ordinary circles when $n$ is odd.
\end{construction}


\section{Tools}

In this section we introduce two basic ingredients of our arguments: circle inversion and tangent lines.


\subsection{Circle inversion}\label{sec:inversion}

A key tool in our proof is circle inversion. 
We quickly introduce it here to fix notation, and we summarize the properties that we use.
See for instance \cite{B} for more details and background.

It seems that inversion was first used in the context of the Sylvester-Gallai theorem by Motzkin \cite{M}.
He used it to prove that for any point set that is not on a line or circle, there is a line or circle containing exactly three of the points.
Inversion was also the main tool in 
each of the papers \cite{E, BB, Z} that established lower bounds on the number of ordinary circles.

\begin{definition}\label{def:inversion}
Let $p \in \R^2$ be a fixed point. The \emph{circle inversion} with center $p=(x_p,y_p)$ is the mapping $I_p:\R^2\backslash\{p\} \rightarrow \R^2\backslash\{p\}$ defined by
$$I_p(x,y) = \left(\frac{x-x_p}{(x-x_p)^2+(y-y_p)^2}+x_p, \frac{y-y_p}{(x-x_p)^2+(y-y_p)^2}+y_p\right)$$
for $(x,y)\neq p$.
\end{definition}

Technically, this is inversion in the circle of radius $1$ around $p$, 
but this radius plays no role in our proof, so we have fixed it to be $1$.
The natural setting for inversion is the \emph{Riemann sphere}, which can be viewed as $\R^2$ together with a point at infinity. 
However, because in our proof we also deal with the projective plane, which is $\R^2$ together with a \emph{line} at infinity, we keep the Riemann sphere out of sight to avoid confusion.

We recall the following basic properties of inversion.
We write $S-x = S\backslash\{x\}$.

\begin{itemize}
\item $I_p$ is its own inverse, i.e., 
$I_p(I_p(q)) = q$ for all $q\in \R^2\backslash \{p\}$;
\item If a line $L$ contains $p$, then $I_p$ maps $L-p$ to $L-p$. If a line $L'$ does not contain $p$, then $I_p$ maps $L'$ to $C-p$ for a circle $C$.
\item If a circle $C$ contains $p$, then $I_p$ maps $C-p$ to a line not containing $p$. If a circle $C'$ does not contain $p$, then $I_p$ maps $C'$ to a circle not containing $p$.
\end{itemize}

We also need the more general fact that inversion maps algebraic curves to algebraic curves.

\begin{lemma}\label{lem:invertcurve}
The inversion $I_p$ maps an algebraic curve $C$ of degree $d$ to an algebraic curve $C'$ of degree at most $2d$.
More precisely,
$C-p$ is mapped bijectively to $C'-p$.
\end{lemma}
\begin{proof}
We can assume that $p=0$.
Let $f(x,y)=0$ be the equation defining $C$.
Since $I_p$ is its own inverse,
every point $(x,y)\in I_p(C-p)$ satisfies the equation
\[ f\left(\frac{x}{x^2+y^2}, \frac{y}{x^2+y^2}\right)=0.\]
Multiplying this equation by $(x^2+y^2)^k$ for the integer $k\leq d$ that eliminates the denominator,
we get a polynomial equation $g(x,y)=0$ of degree at most $2k\leq 2d$.
On the other hand, given a point $q\neq 0$ satisfying $g(x,y)=0$, $r = I_p(q)$ satisfies $f(x,y)=0$, so $q$ lies in the image of $C-p$.
This shows that $C-p$ is mapped to $C'-p$.
\end{proof}

Whether or not $p$ lies on $C'$ depends on the behavior at infinity of $C$.
We just note that for a conic $C$, 
we have $p\in C'$ if $C$ is a parabola or hyperbola,
and $p\not\in C'$ if $C$ is an ellipse.
This follows easily by a continuity argument and the fact that parabolas and hyperbolas have a (real) point at infinity, while ellipses do not.


\subsection{Tangent lines} 

We need a few simple notions from the theory of algebraic curves.
We use the terminology from Fischer \cite{F}, and we temporarily use the projective plane $\proj$ because the definition of the polar curve is more natural there.
For a homogeneous polynomial $F\in \R[x,y,z]$, which we assume to be square-free, we write 
\[V(F) = \{[x:y:z]\in \proj: F(x,y,z)=0\}\] for the algebraic curve that it defines in $\proj$.

\begin{definition}[Polar curve] 
Let $C=V(F)$ be an algebraic curve in $\mathbb{P}\R^2$ not containing a line, 
and $p=[p_1:p_2:p_3] \in \mathbb{P}\R^2$ a fixed point. 
We define the homogeneous polynomial
\[D_{p}(F)=p_{1}{\frac {\partial F}{\partial x_1}}+p_{2}{\frac {\partial F}{\partial x_2}}+p_{3}{\frac {\partial F}{\partial x_3}},\] 
and we call $D_{p}(C)=V(D_{p}(F))$ the \emph{polar curve} of the curve $C$ with respect to $p$. 
\end{definition}

The proof of the following lemma can be found in \cite[Chapter 4]{F}.

\begin{lemma}
\label{lem:polar}
Let $C \subset \proj$ be an algebraic curve not containing a line and $p \in \mathbb{P}\R^2$ a fixed point.
The tangency points on $C$ of the tangent lines from $p$ to $C$ are contained in $C \cap D_{p}(C)$.
\end{lemma}

\begin{corollary}
\label{cor:tangentlinebound}
Let $C$ be an algebraic curve of degree $d\geq 2$ in $\R^2$ and $p \in \R^2$ a fixed point.
At most $d(d-1)$ lines through $p$ are tangent to $C$ (or contained in $C$).
\end{corollary}
\begin{proof}
We can assume that the polynomial defining $C$ has at most $d-2$ linear factors, 
since otherwise the curve is a union of lines and the bound follows immediately.
Remove the linear factors and let $C'$ be the curve defined by the remaining polynomial; it does not contain a line and has degree $e\geq 2$.
Then $C$ is the union of $C'$ and $d-e$ lines.

Let $C''$ be the projective closure (see \cite[Chapter 2]{F}) of $C'$, which is a curve in $\proj$, also of degree $e$.
Applying B\'ezout's theorem (see \cite[Chapter 2]{F}) to $C''$ and $D_{p}(C'')$ gives 
\[|C'' \cap D_{p}(C'')| \leq \deg C''\cdot \deg D_{p}(C'')=e(e-1).\] 
It follows by Lemma \ref{lem:polar} that there are at most $e(e-1)$ tangency points on $C'$ from the tangent lines passing through $p$, so there are at most $e(e-1)$ tangent lines from $p$ to $C'$.

The same bound then holds in $\R^2$ for tangent lines from $p$ to $C$ that are not contained in $C$.
Together with the $d-e$ lines contained in $C$, this gives the bound $d(d-1)$.
\end{proof}


\section{Ordinary lines avoiding a fixed point} 
\label{sec:greentao}

In this section we introduce the structure theorem for sets with few ordinary lines of Green and Tao, 
and we deduce from it a structure theorem for sets with few ordinary lines \emph{avoiding some fixed point}.

\subsection{The Green-Tao structure theorem}

We use the notation $[x:y:z]$ for points in the projective plane $\proj$.
We consider points of the form $[x:y:0]$ as points on the line at infinity,
and we write $\mathbf{0}$ for the ``origin'' of $\proj$, by which we mean the point $[0:0:1]$.
We say that a set $A$ in $\R^2$ is \emph{projectively equivalent} to a set $B$ in $\proj$ if, after embedding $\R^2$ in $\proj$ by $(x,y)\mapsto [x:y:1]$, there is a bijective projective transformation that maps $A$ to $B$.
We denote this by $A\approx B$.

The following point sets are crucial in the proof of Green and Tao \cite{GT}, and will also be central to our proof.
In \cite{CM}, these constructions are ascribed to B\"or\"oczky. Further information as well as wonderful pictures can be found in \cite{GT}.

\begin{construction}[B{\"o}r{\"o}czky examples]\label{constr:bor}
For every $m \in \mathbb{N}$, we define 
\begin{align*}
X_{2m} = &\{ [ \cos\frac{2\pi j}{m}:\sin\frac{2\pi j}{m}:1] :0 \le j < m\} \cup \{[-\sin\frac{\pi j}{m}: \cos\frac{\pi j}{m}:0] :0 \le j < m \},
\end{align*}
which consists of $m$ points on a unit circle and $m$ points on the line at infinity in $\mathbb{P}\R^2$. 
A \emph{B{\"o}r{\"o}czky example} is either a set in $\mathbb{P}\R^2$ that is projectively equivalent to some $X_{2m}$,
or a set in $\R^2$ that is projectively equivalent to some $X_{2m}$.

A B{\"o}r{\"o}czky example $P$ consists of $m$ points on a conic $C$ and $m$ points on a line $L$.
As shown in \cite[Proposition 2.1]{GT},
every ordinary line of $P$ is a tangent line of $C$.
Moreover, for every point $p$ of $P$ on $C$, the tangent line to $C$ at $p$ is an ordinary line of $P$ (i.e., the point where the tangent line intersects $L$ is also in $P$).
This shows that a B{\"o}r{\"o}czky example with $2m$ points has exactly $m$ ordinary lines.
Another consequence is that a line through two points of $P$ on $C$ must intersect $L$ in a point of $P$, since otherwise it would be an ordinary line that is not tangent to $C$.
\end{construction}

We state the main result from Green and Tao \cite{GT}, which gives a detailed picture of the structure of any set with few ordinary lines.
We refer to \cite{GT} for definitions of the terms in type $(ii)$ of the theorem, including the group structure that is referred to.

\begin{theorem}[Green-Tao]
\label{thm:GT1}
For every $K\in\R$ there exists an $N_K \in \mathbb{N}$ such that the following holds for any set $P$ of $n \geq N_K$ points in $\proj$.
If $P$ determines fewer than $Kn$ ordinary lines, then after a projective transformation, $P$ differs by at most $O(K)$ points from a set of one of the following types:
\begin{itemize}
\item[$(i)$] $n-O(K)$ points on a line;
\item[$(ii)$] A coset $H\oplus g$ of a subgroup $H$ of $n\pm O(K)$ nonsingular points of an irreducible cubic curve, where $g$ is a point on the curve satisfying $3g\in H$;
\item[$(iii)$] The B{\"o}r{\"o}czky example $X_{2m}$, for some $m=\frac {n}{2}\pm O(K)$.
\end{itemize}
\end{theorem}


\subsection{\texorpdfstring{Non-$q$ ordinary lines}{Non-q ordinary lines}}

In the proof of Theorem \ref{thm:main}, we apply a circle inversion in a point of the point set, find an ordinary line in the inverted set, and then observe that this line is inverted back to a circle that also contains the center of inversion, so this circle is ordinary.
However, this only works if the ordinary line in the inverted set does not pass through the center of inversion. This leads to the following notion.

\begin{definition}
Let $P$ be a set of $n$ points in $\R^2$ and $q$  a fixed point in $\R^2$ that is not in $P$. 
We call an ordinary line of $P$ a \emph{non-$q$ ordinary line} if it does not pass through $q$. We denote the number of non-$q$ ordinary lines of $P$ by $ol_q(P)$.
\end{definition}

We now use Theorem \ref{thm:GT1}, as well as some observations from its proof in \cite{GT}, 
to deduce a structure theorem for sets with few non-$q$ ordinary lines. 
For brevity, we use the notation $S+x = S\cup \{x\}$ when we add a point to a set, 
and $S-x = S\backslash\{x\}$ when we remove a point from a set.

\begin{theorem}
\label{thm:addrem}
Suppose $P$ is a set of $n$ non-collinear points in $\R^2$. 
For a fixed point $q \notin P$, the number of non-$q$ ordinary lines of $P$ is at least $n-O(1)$, unless $P$ satisfies one of the following descriptions\footnote{When we write for instance $P\approx X_{n+1}-q$, we should really be writing $P\approx X_{n+1}-\varphi(q)$, where $\varphi$ is the projective transformation involved.
We are omitting this to avoid further complicating the already long statement of the theorem.}: 
\begin{itemize}
\item[$(a)$] $n$ is odd,  $P\approx X_{n+1}-q$, 
and  $ol_q(P)= \frac{1}{2}n-O(1)$;
\item[$(b)$] $n$ is odd,  $P\approx X_{n+1}-p$ with $p\neq q$, and $ol_q(P)= \frac{3}{4}n-O(1)$;
\item[$(c)$] $n$ is odd, $P\approx X_{n-1}+\mathbf{0}$, and $ol_q(P)= \frac{3}{4}n-O(1)$;
\item[$(d)$] $n$ is even, $P\approx X_n$, and $ol_q(P)=\frac{1}{2}n-O(1)$; 
\item[$(e)$] $n$ is even, $P\approx X_{n+2}-p-q$, and $ol_q(P)= \frac{3}{4}n-O(1)$;
\item[$(f)$] $n$ is even, 
$P\approx X_n+\mathbf{0}-q$, and $ol_q(P)= \frac{3}{4}n-O(1)$.
\end{itemize} 
\end{theorem}
\begin{proof}
We can assume that $P$ has less than $2n$ ordinary lines,
since otherwise the the fact that there are less than $n$ ordinary lines containing $q$ implies that there are at least $n$ non-$q$ ordinary lines.
By adjusting the constants in the $O(1)$ terms, 
we can ensure that $n$ is large enough to let us apply Theorem \ref{thm:GT1}.
We apply Theorem \ref{thm:GT1} with $K=2$ and deduce that after a projective transformation, $P$ differs by at most $O(1)$ points from a set of one of the three types in Theorem \ref{thm:GT1}.

If $P$ is type $(i)$, then it has $n-O(1)$ points on a line $L$, but there is a point $p\in P\setminus L$. 
Then there are $n-O(1)$ lines through $p$ and a point of $P$ on $L$.
At most $O(1)$ of these pass through a third point of $P$, so $P$ has $n-O(1)$ ordinary lines containing $p$.
At most one of them can pass through $q$,
which implies $ol_q(P)\ge n-O(1)$ in this case.

If $P$ is type $(ii)$,
then it differs by $O(1)$ points from a coset $H$ of a subgroup\footnote{We refer to \cite{GT} for definitions, because this paragraph is the only point in our paper where we consider the group structure on an irreducible cubic.} of an irreducible cubic $C$.
Moreover, it can be deduced from \cite{GT}
that $P$ has at least $n-O(1)$ ordinary lines that are tangent lines to $C$.
Specifically, in \cite[Proposition 2.6]{GT} it is shown that for a coset $H$ of a subgroup of $C$, the ordinary lines are exactly the tangent lines to $C$ at $|H|-O(1)$ points of the coset.
When a point is added to or removed from $H$ to obtain $P$, 
an ordinary line may be broken,
either because a third point is added or because one of the two points is removed.
However, adding or removing a point breaks at most $O(1)$ of the ordinary lines that are tangent to $C$, 
since by Corollary \ref{cor:tangentlinebound} a fixed point is contained in at most six tangent lines to a cubic curve.
We have established that $P$ has at least $n-O(1)$ ordinary lines that are tangent lines to the curve $C$.
Again by Corollary \ref{cor:tangentlinebound}, 
at most six of these lines pass through $q$,
which shows that $P$ has at least $n-O(1)$ non-$q$ ordinary lines.

Finally, we treat the case where $P$ is type $(iii)$. 
Since ordinary lines (and non-$q$ ordinary lines) are preserved under projective transformations, 
we can assume that $P$ itself differs by at most $O(1)$ points from $X_{2m}$. 
First consider a B{\"o}r{\"o}czky example $X_{2m}$, consisting of $m$ points on a unit circle $S$ and $m$ points on the line $L$ at infinity.
As mentioned, $X_{2m}$ has $m$ ordinary lines, 
all of which are tangent lines to $S$. 

If we add to $X_{2m}$ a point $p$ other than the origin, 
then at least $m-O(1)$ ordinary lines are created, 
all of which pass through $p$.
For a point not on $L$, this is \cite[Corollary 7.6]{GT}.
For a point on $L$, it follows from the fact that any line through two points of $X_{2m}\cap S$ intersects $L$ in a third point of $X_{2m}$,
which implies that any line through $p\in L\backslash X_{2m}$ and a point of $X_{2m}\cap S$ contains no other point of $X_{2m}$.
By Corollary \ref{cor:tangentlinebound}, $p$ is contained in at most two of the ordinary lines of $X_{2m}$ (which are tangent to $S$), 
so adding $p$ breaks at most two of the ordinary lines of $X_{2m}$.
Therefore, $X_{2m}+p$ has at least $2m-O(1)$ ordinary lines, 
of which  $m-O(1)$ are tangent lines to $S$,
 and $m-O(1)$ pass through $p$. 
 By Corollary \ref{cor:tangentlinebound}, at most two of the tangent lines pass through $q$, 
 and assuming $p\neq q$, at most one of the lines through $p$ hits $q$.
 Thus
$$ol_q( X_{2m}+p)\ge 2m-O(1).$$
Since $q\not\in P$, $q$ itself is not one of the points that were added to $X_{2m}$ to obtain $P$.
Thus we have covered the cases where a single point other than the origin is added.

Next we consider what happens if we remove a point from $X_{2m}$.
Since the ordinary lines of $X_{2m}$ are tangent lines to $S$, 
and since through any point there are at most two such tangent lines by Corollary \ref{cor:tangentlinebound},
removing a point from $X_{2m}$ breaks at most two ordinary lines.
To see how many ordinary lines are created by removing a point of $X_{2m}$,
observe that at least $\frac{1}{2}m-O(1)$ of the lines through a point of $X_{2m}$ contain three points of $X_{2m}$.
Indeed, a point on $X_{2m}\cap L$ is contained in at least $\frac{1}{2}(m-2)$ lines that contain two points of $X_{2m}\cap S$,
while a point of $X_{2m}\cap S$ is contained in at least $m-1$ lines through another point of $X_{2m}\cap S$, each of which intersects $L$ in a third point of $X_{2m}$.
It follows that removing a point creates at least $\frac{1}{2}m-O(1)$ ordinary lines, all of which pass through the removed point. 
Note that a removed point could be $q$ itself, in which case the new ordinary lines are not non-$q$ ordinary lines.

Now consider adding or removing several points.
The above shows that after we add one point $p$ (other than the origin), there are $2m-O(1)$ non-$q$ ordinary lines, which are either tangent to $S$ or pass through $p$.
Any other point $r$ lies on at most two tangent lines to $S$ and on one line through $q$, 
so adding or removing any other point breaks $O(1)$ non-$q$ ordinary lines.
Similarly, if  we remove two points (not including $q$), 
then the number of non-$q$ ordinary lines is at least
$$\left(m-O(1)\right)+2\cdot\left(\frac{1}{2}m-O(1)\right)=2m-O(1).$$
Any subsequent removal breaks at most $O(1)$ of these ordinary lines.

If we add the origin (which is the center of $S$) to $X_{2m}$, no ordinary lines are broken, since the origin is not contained in any tangent line to $S$.
By \cite[Proposition 2.1]{GT}, adding the origin creates at least $\frac{1}{2}m-O(1)$ ordinary lines, all of which pass through the origin.
Thus if we add the origin  and remove at least one point other than $q$, 
we again get at least $2m-O(1)$ non-$q$ ordinary lines. 

As a result, if $P$ has less than $2m-O(1)=n-O(1)$ non-$q$ ordinary lines, 
it must be obtained from $X_{2m}$ by:
making no change (type $(d)$); 
removing $q$ (type $(a)$);
removing a point other than $q$ (type $(b)$);
removing two points, one of which equals $q$ (type $(e)$);
adding the origin (type $(c)$);
or adding the origin and removing $q$ (type $(f)$).
For each case, the parity of $n$ then follows from the fact that $X_{2m}$ has an even number of points, 
and the number of non-$q$ ordinary lines follows from the observations above.
\end{proof}


\section{Even point sets}\label{sec:even}

In this section we prove Theorem \ref{thm:main} for point sets of even size; point sets of odd size are treated in the next section.
Let $n$ be an even positive integer and suppose $P$ is a set of $n$ points in $\mathbb{R}^2$ that are not contained in a line or a circle. 
For every point $p\in P$, we apply the inversion $I_p$, and we consider the non-$p$ ordinary lines of $I_p(P-p)$, 
which correspond to ordinary circles of $P$.
By Theorem \ref{thm:addrem}, we have the lower bound $ol_p(I_p(P-p))\ge \frac{1}{2}n -O(1)$ for all $p\in P$, and this bound can be attained for sets of type $(a)$.
However, the resulting lower bound on ordinary circles would be $\frac{1}{6}n^2 - O(n)$, which is not optimal.
To obtain a better overall bound, we will use a careful case analysis to show that it is not possible to obtain sets of type $(a)$ after all inversions.

Let $oc(P)$ be the total number of ordinary circles of $P$ and $oc_p(P)$ the number of ordinary circles of $P$ containing $p$. 

\bigskip\noindent
{\bf Case I:} {\it For every $p \in P$, we have $ol_p(I_p(P-p))\ge \frac {3}{4}n -O(1)$.}\\
Every non-$p$ ordinary line of $I_p(P-p)$ corresponds to an ordinary circle containing $p$.
Since each ordinary circle contains exactly three points of $P$, by double counting we have:
\begin{align}\label{eq:allthreefourth}
3oc(P)
=\sum_{p \in P} oc_p(P) 
\geq n\cdot\left(\frac{3}{4}n-O(1)\right)
=\frac{3}{4}n^2-O(n).
\end{align}
This gives the desired bound in this case.

\bigskip\noindent
{\bf Case II:} {\it  There exists $q \in P$ for which we do not have $ol_q(I_q(P-q))\geq \frac {3}{4}n -O(1)$.}\\
In this case we will prove that roughly half the points give at least $n-O(1)$ ordinary circles each.
Together with the weaker overall bound $\frac{1}{2}n-O(1)$ for the other points, this still leads to the desired total.

Since $|I_q(P-q)|=n-1$ is odd, 
the set $I_q(P-q)$ must be type $(a)$ by Theorem \ref{thm:addrem}, so $ol_q(I_q(P-q))=\frac{1}{2}n-O(1)$,
and $I_q(P-q)$ is a B\"or\"oczky example 
minus the point $q$.
Hence the points of $I_q(P-q)+q$ lie on a line $L$ and a conic $C$ (and the line and conic are disjoint).
We distinguish between the subcase where $q\in C$ and the subcase where $q\in L$.

\bigskip\noindent
{\bf Subcase II(a):} {\it The point $q$ lies on the conic $C$.}\\
The line $L$ contains $\frac{1}{2}n$ points of $I_q(P-q)$ and the conic $C$ contains $\frac{1}{2}n-1$ points of $I_q(P-q)$ (not counting $q$).
Applying $I_q$ again (so reversing the inversion), the line $L$ is mapped to a circle $S$, 
which contains $\frac{1}{2}n+1$ points of $P$ (including $q$).
By Lemma \ref{lem:invertcurve}, the conic\footnote{For convenience, we refer to a set like $C-p$ as a curve, even though it may actually be a curve minus a point (the center of inversion); otherwise we would need a lot of unpleasant notation.} $C-p$ is mapped to an algebraic curve $C'$ of degree at most $4$.
The circle $S$ and the curve $C'$ are disjoint or intersect only in $q$.
As remarked after Lemma \ref{lem:invertcurve}, 
if $C$ is an ellipse, then $C'$ does not contain $q$, 
and thus contains $\frac{1}{2}n-1$ points of $P$;
if $C$ is a parabola or hyperbola,
then $C'$ does contain $q$, 
and thus contains $\frac{1}{2}n$ points of $P$.

We consider $p\in P\cap C'$, not equal to $q$, and apply $I_p$.
The set $I_p(P-p)$ has $n-1$ points,
with $\frac{1}{2}n+1$  on the circle $I_p(S)$ and $\frac{1}{2}n-2$ or $\frac{1}{2}n-1$ points (excluding $p$) on the algebraic curve $I_p(C'-p)$ of degree at most $8$.
If $I_p(P-p)$ were one of the types in Theorem \ref{thm:addrem}, then the two curves in those types (the line and conic of $X_{2m}$) would have to equal the two curves $I_p(S)$ and $I_p(C'-p)$, 
since by B\'ezout's inequality, any two of these curves intersect in at most 16 points, and $n$ is assumed to be much larger.
It then follows that $I_p(P-p)$ cannot be one of the three types of odd point sets listed in Theorem \ref{thm:addrem},
since in each of those the difference between the number of points on the two curves is $1$,
while here this difference is $3$ or $2$.
Hence the number of non-$p$ ordinary lines of $I_p(P-p)$ is at least $n-O(1)$. 

For points $r\in P\cap S$, Theorem \ref{thm:addrem} still gives us $oc_r(P) \geq \frac{1}{2}n-O(1)$.
Thus the total number of ordinary circles satisfies
\begin{align}\label{eq:halfnandn}
\begin{split}
3oc(P)&\geq\sum_{r \in P\cap S} oc_r(P)+\sum_{p \in P\cap (C'-q)} oc_p(P)\\
&\ge \left(\frac{1}{2}n+1\right) \left(\frac{1}{2}n-O(1)\right)+\left(\frac{1}{2}n-1\right) \bigg(n-O(1)\bigg)=\frac{3}{4}n^2-O(n).
\end{split}
\end{align}

\bigskip\noindent
{\bf Subcase II(b):} {\it The point $q$ lies on the line $L$.}\\
In this case, $P$ lies on the line $L$ (note that now $I_q(L) = L$) and an algebraic curve $C''$ of degree at most $4$, with $L$ and $C''$ disjoint or intersecting only in $q$. 
The line $L$ contains $\frac{1}{2}n$ points of $P$ 
and the curve $C''$ contains $\frac{1}{2}n$ or $\frac{1}{2}n+1$ points of $P$.
We will show that for most points $r\in P\cap C''$ we have $oc_r(P) \geq n-O(1)$, which will again lead to the desired total.

Suppose that for a point $r\in P\cap C''$, not equal to $q$, we do not have $ol_r\left(I_r(P-r)\right) \geq n-O(1)$, 
so $I_r(P-r)$ must be one of the types in Theorem \ref{thm:addrem}.
The set $I_r(L)$ is a circle and $I_r(C''-r)$ is an algebraic curve of degree at most $8$.
If these curves were not disjoint, 
then $I_r(P-r)$ could not be any of the types in Theorem \ref{thm:addrem},
so $I_r(L)$ and $I_r(C''-r)$ are disjoint.
In particular, this implies that $C''$ contains $\frac{1}{2}n$ points of $P$.
Therefore, the curve $I_r(C''-r)$ contains $\frac{1}{2}n-1$ points,
and the circle $I_r(L)$ contains $\frac{1}{2}n$ points of $I_r(P-r)$ (not counting $r$).
Note that $r$ and $I_r(q)$ both lie on the circle $I_r(L)$.

By Theorem \ref{thm:addrem},
$I_r(P-r)$ must be type $(a)$, $(b)$, or $(c)$, 
with the two curves in each type equalling $I_r(L)$ and $I_r(C''-r)$.
It cannot be type $(a)$, since $r$ lies on the curve with more points, so it cannot be equal to the removed point.
The set $I_r(P-r)$ also cannot be type $(c)$, 
since $I_r(P-r)$ has no points outside the two curves.
Thus it is type $(b)$, 
so $I_{r}(P-r)$ is projectively equivalent to $X_n$ minus a point not equal to $r$.
Note that $I_r(L)$ is not a line, 
so $I_{r}(C''-r)$ must be the line of the B{\"o}r{\"o}czky example, 
which tells us that $C''$ is in fact a circle.
Also note that, since $r$ is not part of the B{\"o}r{\"o}czky example, the line spanned by $r$ and $I_r(q)$ does not intersect the line $I_{r}(C''-r)$ in a point of $I_{r}(P-r)$, since otherwise it would be an ordinary line in the B{\"o}r{\"o}czky example that is not tangent to the conic.
In terms of $P$, this means that the line spanned by $r$ and $q$ does not intersect the circle $C''$ in a point of $P$ other than $r$. 

Recall that, under the assumption of Case II, the set $I_q(P-q)+q$ is a B\"or\"oczky example,
with its points on the line $L$ and the conic $C$.
As explained in Construction \ref{constr:bor}, 
a B\"or\"oczky example has the following feature.
For every $s\in C\cap I_q(P-q)$, the line spanned by $s$ and $q\in L$ is either tangent to the conic $C$ at $s$, 
or intersects it in a point of $I_q(P-q)$ other than $s$. 
By Corollary \ref{cor:tangentlinebound}, there are at most two lines through $q$ that are tangent to $C$. 
In terms of $P$, 
this means that at most two points $r\in P\cap C''$ have the property that the line spanned by $r$ and $q$ does not intersect $C''$ in a point of $P$ other than $r$.
By the previous two paragraphs, this implies that there are at most two $r\in P\cap C''$ for which we do not have $ol_r\left(I_r(P-r)\right) \geq n-O(1)$.
By essentially the same counting as in \eqref{eq:halfnandn}, we get
$oc(P)\ge\frac{1}{4}n^2-O(n)$. 
 This finishes the proof of Theorem \ref{thm:main} for even values of $n$.


\section{Odd point sets} 

In this section, we prove Theorem \ref{thm:main} for odd values of $n$. 
The proof uses similar arguments and a similar case structure as in Section \ref{sec:even}, but in a different order.
Let $n$ be an odd positive integer and suppose $P$ is a set of $n$ points in $\mathbb{R}^2$ which are not contained in a line or a circle. 

\bigskip\noindent
{\bf Case I:} {\it For every $p \in P$, we have $ol_p(I_p(P-p))\ge \frac {3}{4}n -O(1)$.}\\
We get the desired bound by the same calculation \eqref{eq:allthreefourth} as in Case I in Section \ref{sec:even}.

\bigskip\noindent
{\bf Case II:} {\it  There exists $q \in P$ for which we do not have $ol_q(I_q(P-q))\geq \frac {3}{4}n -O(1)$.}\\
Since $|I_q(P-q)|=n-1$ is even,
Theorem \ref{thm:addrem} implies that 
 $I_q(P-q)$ is type $(d)$, i.e., it is a B\"or\"oczky example.
 It has $\frac{1}{2}(n-1)$ points on a line $L$ and the same number of points on a disjoint conic $C$. 

\bigskip\noindent
{\bf Subcase II(a):} {\it The point $q$ lies on the line $L$.}\\
We deduce that $P$ consists of $\frac{1}{2}(n-1) + 1$ points on the line $L$ and $\frac{1}{2}(n-1)$ or $\frac{1}{2}(n-1)+1$ points on an algebraic curve $C'$ of degree at most $4$,  with $L $ and $C'$ disjoint or intersecting only in $q$. 
Fix $p\in P\cap C'$, not equal to $1$,  and apply $I_p$ to $P-p$. 
Then $I_p(P-p)$ consists of $\frac{1}{2}(n-1)+1$ points on the circle $I_p(L)$, and $\frac{1}{2}(n-1) -1$ or $\frac{1}{2}(n-1)$ points on the algebraic curve $I_p(C'-p)$ of degree at most $8$. 
Note that $p$ lies on the circle $I_p(L)$.
The set $I_p(P-p)$ cannot be type $(d)$ since it has a different number of points on the two curves,
it cannot be type $(e)$ because $p$ lies on the curve with more points, 
and it cannot be type $(f)$ since all its points lie on the two curves.
Thus, by Theorem \ref{thm:addrem}, we have $ol_p(I_p(P-p))\ge n-O(1)$.
On the other hand, by the same theorem, for every $p \in P\cap L$ we have $ol_p(I_p(P-p))\ge \frac{1}{2}n-O(1)$.
As a result, using the calculation \eqref{eq:halfnandn} from Section \ref{sec:even}, we get $oc(P)\ge \frac{1}{4}n^2-O(n)$, which completes the proof in the first subcase.

\bigskip\noindent
{\bf Subcase II(b):} {\it The point $q$ lies on the conic $C$.}\\
The set $P$ has $\frac{1}{2}(n-1)+1$ points  on a circle $S$ and $\frac{1}{2}(n-1)$ or $\frac{1}{2}(n-1)+1$ points on an algebraic curve $C''$ of degree at most $4$, with $S$ and $C''$ disjoint or intersecting only in $q$.
If for every point $p \in P\cap C''$, not equal to $q$, we have $ol_p(I_p(P-p)) \ge n-O(1)$, then again by the calculation \eqref{eq:halfnandn}, we get $oc(P)\ge\frac{1}{4}n^2-O(n)$.
Thus we can assume that there exists $r \in P\cap C''$, not equal to $q$, 
for which $ol_r(I_r(P-r)) \geq n-O(1)$ does not hold.
This immediately implies that $S$ and $C''$ are disjoint, and that $C''$ contains $\frac{1}{2}(n-1)$ points of $P$.
We will show that, given this assumption, for all but two points $p\in P$ we have 
$ol_p(I_p(P-p)) \ge \frac{3}{4}n-O(1)$.
By the calculation \eqref{eq:allthreefourth}, this gives the required bound on $oc(P)$ and completes the proof.

The set $I_r(P-r)$ consists of $\frac{1}{2}(n-1)+1$ points on the circle $I_r(S)$ and $\frac{1}{2}(n-1)-1$ points on the algebraic curve $I_r(C''-r)$. 
Therefore, $I_r(P-r)$ cannot be type $(d)$, 
so Theorem \ref{thm:addrem} tells us that $ol_r(I_r(P-r)) =\frac{3}{4}n-O(1)$.
Moreover, $I_r(C''-r)+r$ must be a line, which implies that $C''$ is also a line. 
Since $I_r(P-r)$ has no point outside the two curves, it cannot be type $(f)$, so it must be type $(e)$.
Thus $I_r(P-r)$ is projectively equivalent to $X_{n+1}$ minus two points, one of which is $r$.

The same reasoning shows that for every $p \in P\cap C''$,
$I_p(P-p)$ is not type $(d)$, 
so we have $ol_p(I_p(P-p)) \ge \frac{3}{4}n-O(1)$ for every $p \in P\cap C''$.

Finally, we show that there are at most two $s\in P\cap S$ for which we do not have 
$ol_s(I_s(P-s)) \geq \frac{3}{4}n-O(1)$.
Suppose that $s\in P\cap S$ does not have $ol_s(I_s(P-s)) \geq \frac{3}{4}n-O(1)$.
Then by Theorem \ref{thm:addrem},
$I_s(P-s)$ is type $(d)$, 
so it is a B\"or\"oczky example consisting of $\frac{1}{2}(n-1)$ points on the line $I_s(S-s)$ and $\frac{1}{2}(n-1)$ points on the disjoint circle $I_s(C'')+s$.
The structure of the B\"or\"oczky example implies that the line spanned by $s\in I_s(C'')+s$ and $I_s(r)\in I_s(C'')$ does not pass through any points of $I_s(P-s)$ other than $I_s(r)$, which implies that the line spanned by $s\in S$ and $r\in C''$ does not intersect $P$ in any other point.
On the other hand, looking back at the point set $I_r(P-r)$, the line spanned by $I_r(s)$ and $r$ is either tangent to the circle $I_r(S)$ at $I_r(s)$, or it passes through a point of $I_r(P-r)$ other than $I_r(s)$.
By Corollary \ref{cor:tangentlinebound}, since there are at most two such tangent lines, there are at most two such points $s$. 

This finishes the proof for the odd case.


\end{document}